\documentclass[12pt]{amsart}
\usepackage{amsmath}
\usepackage{amscd}
\usepackage{amsfonts}
\usepackage{amsthm}
\usepackage[usenames,dvipsnames]{color}
\usepackage{enumerate}
\usepackage{mwe} 
\usepackage{subcaption}
\usepackage{graphicx}
\usepackage{float}
\usepackage{afterpage}
\usepackage{wrapfig}
\usepackage{lscape}
\usepackage[figuresright]{rotating}
\usepackage{epstopdf}
\usepackage{fancyhdr}

\usepackage{epsfig}

\newtheorem{theorem}{Theorem}[section]

\newtheorem{lemma}[theorem]{Lemma}

\theoremstyle{definition}
\newtheorem{definition}[theorem]{Definition}
\newtheorem{remark}{Remark}

\def \epsilon {\varepsilon}

\def \phi {\varphi}

\def \J {{\mathcal J}}

\def \Jm {{\mathcal J}_m}

\def \Jn {{\mathcal J}_n}

\def \F {{\mathcal F}}

\def \Fm {{\mathcal F}_m}

\def \Fn {{\mathcal F}_n}

\def \S {{\mathcal S}}

\def \C {\mathbb C}

\def \chat {{\widehat{\mathbb C}}}

\def \R {\mathbb R}

\def \N {\mathbb N}

\def \nochat {{\mathbb N}_0 \times {\widehat {\mathbb C}}}

\def \Pm {\{P_m \}_{m=1}^\infty}

\def \HD {{\mathrm{HD}}}

\title[HNUP Non-Autonomous Julia Sets]{Hereditarily non Uniformly Perfect non-Autonomous Julia Sets}

\author{Mark Comerford, Rich Stankewitz, Hiroki Sumi}



\email{mcomerford@math.uri.edu}
\email{rstankewitz@bsu.edu}
\email{sumi@math.h.kyoto-u.ac.jp}

\begin{document}

\maketitle
\centerline{\scshape Mark Comerford}
\medskip
{\footnotesize
 \centerline{Department of Mathematics}
   \centerline{University of Rhode Island}
   \centerline{5 Lippitt Road, Room 102F}
   \centerline{Kingston, RI 02881, USA}
} 

\medskip

\centerline{\scshape Rich Stankewitz }
\medskip
{\footnotesize
 \centerline{Department of Mathematical Sciences}
   \centerline{Ball State University}
   \centerline{Muncie, IN 47306, USA}
} 

\medskip

\centerline{\scshape Hiroki Sumi }
\medskip
{\footnotesize
 \centerline{Course of Mathematical Science}
 \centerline{Department of Human Coexistence}
 \centerline{Graduate School of Human and Environmental Studies}
   \centerline{Kyoto University}
   \centerline{Yoshida-nihonmatsu-cho, Sakyo-ku}
   \centerline{Kyoto 606-8501, Japan}
} 

\bigskip

\begin{abstract}Hereditarily non uniformly perfect (HNUP) sets were introduced by Stankewitz, Sugawa, and Sumi in \cite{SSS} who gave several examples of such sets based on Cantor set-like constructions using nested intervals. We exhibit a class of examples in non-autonomous iteration where one considers compositions of polynomials from a sequence which is in general allowed to vary. In particular, we give a sharp criterion for when Julia sets from our class will be HNUP and we show that the maximum possible Hausdorff dimension of $1$ for these Julia sets can be attained. The proof of the latter considers the Julia set as the limit set of a non-autonomous conformal iterated function system and we calculate the Hausdorff dimension using a version of Bowen's formula given in the paper by Rempe-Gillen and Urb\'{a}nski \cite{RU}.
\end{abstract}

\section{Introduction}

Our paper is concerned with non-autonomous iteration of complex polynomials. This subject was started by Fornaess and Sibony \cite{FS} in 1991 and by Sester, Sumi and others who were working in the closely related area of skew-products \cite{Ses, Sumi1, Sumi2, Sumi3, Sumi4}. There is also an extensive literature in the real variables case which is mainly focused on topological dynamics, chaos, and difference equations, e.g. \cite{Balibrea, CL}.  A key idea in our work is linking non-autonomous iteration to iterated function systems, most particularly Moran-set constructions. A good exposition on the classical version of this can be found in \cite{Wen}, but the non-autonomous version we make use of is described in the paper of Rempe-Gillen and Urb\'anski \cite{RU}.

We begin with the basic definitions we need in order to state the main theorems of this paper. In the following sections, we then prove these theorems, together with some supporting results and make a few concluding remarks.

\subsection{Polynomial Sequences}

Let $\Pm$ be a sequence of polynomials where each $P_m$ has degree $d_m \ge 2$. For each $0 \le m$, let $Q_m$ be the composition $P_m \circ \cdots \cdots \circ P_2 \circ P_1$ (where for convenience we set $Q_0 = Id$)
and, for each $0 \le m \le n$, let $Q_{m,n}$ be the composition $P_n \circ \cdots \cdots \circ P_{m+2} \circ P_{m+1}$ (where we let each $Q_{m,m}$ be the identity). Such a sequence can be thought of in terms of the sequence of iterates of a skew product on $\chat$ over the non-negative integers $\N_0$ or, equivalently, in terms of the sequence of iterates of a mapping $F$ of the set $\nochat$ to itself, given by $F(m, z) := (m+1, P_{m+1}(z))$.
Let the degrees of these compositions $Q_m$ and $Q_{m, n}$ be $D_m$ and $D_{m,n}$ respectively so that $D_m = \prod_{i=1}^m d_i$, $D_{m,n} = \prod_{i=m+1}^n d_i$.

For each $m \ge 0$ define the \emph{$m$th iterated Fatou set} $\Fm$ by
\[ \Fm = \{z \in \chat : \{Q_{m,n}\}_{n=m}^\infty \;
\mbox{is normal on some neighbourhood of}\; z \}\]
where we take our neighbourhoods with respect to the spherical topology on $\chat$. We then define the \emph{$m$th iterated Julia set} $\Jm$ to be the complement $\chat \setminus \Fm$. At time $m =0$ we call the corresponding iterated Fatou and Julia sets simply the \emph{Fatou} and \emph{Julia sets} for our sequence and designate them by $\F$ and $\J$ respectively.

One can easily show that the iterated Fatou and Julia sets are completely invariant in the following sense.

\begin{theorem}\label{ThmCompInvar}
For each $0 \le m \le n$, $Q_{m,n}(\Fm) = \Fn$ and $Q_{m,n}(\Jm) = \Jn$ with components of $\Fm$ being mapped surjectively onto components of $\Fn$.
\end{theorem}

An important special case is when we have an integer $d \ge 2$ and real numbers $M \ge 0$, $K \ge 1$ for which our sequence $\Pm$ is such that
\[P_m(z) = a_{d_m,m}z^{d_m} + a_{d_m-1,m}z^{d_m-1} + \cdots \cdots +
a_{1,m}z + a_{0,m}\]
is a polynomial of degree $2 \le d_m \le d$ whose coefficients satisfy
\[\frac{1}{K} \le |a_{d_m,m}| \le K,\quad m \ge 1, \:\:\quad \quad |a_{k,m}| \le M,\quad m \ge 1,\:\:\: 0 \le k \le d_m -1. \]
Such sequences are called \emph{bounded sequences of polynomials} or simply \emph{bounded sequences} (see e.g. \cite{Com1, Com2}), this definition being a slight generalization of that originally made by Fornaess and Sibony in \cite{FS} who considered bounded sequences of monic polynomials.

In what follows, for $z \in \C$ and $r > 0$, we use the notation ${\mathrm D}(z, r)$ for the open disc with centre $z$ and radius $r$, while the corresponding closed disc and boundary circle will be denoted by
$\overline{\mathrm D}(z,r)$ and ${\mathrm C}(z, r)$ respectively. For $z \in \C$ and $0 < r< R$, we use ${\mathrm A}(z, r,R)$ for the round annulus $\{w:r < |w-z| < R\}$ with centre $z$, inner radius $r$, and outer radius $R$, while we use $\overline {\mathrm A}(z, r,R)$
for the corresponding closed annulus. 

\subsection{Hereditarily non Uniformly Perfect Sets}\label{SectHNUP}

We call a doubly connected domain $A$ in $\C$ that can be conformally mapped onto a true (round) annulus $\mathrm{A}(z,r,R)$, for some $0<r<R$, a \emph{conformal annulus} with the \emph{modulus} of $A$ given by $\textrm{mod }A=\log(R/r)$, noting that $R/r$ is uniquely determined by $A$ (see, e.g., the version of the Riemann mapping theorem for multiply connected domains in \cite{Ahl}).

\begin{definition} \label{sepann}
A conformal annulus $A$ is said to
\emph{separate} a set $F \subset \C$ if $F \cap A = \emptyset$ and $F$ intersects both components of $\C \setminus
A$.
\end{definition}

\begin{definition} \label{updef}
A compact subset $F \subset \C$ with two or more points is \emph{uniformly perfect} if
there exists a uniform upper bound on the moduli of all conformal annuli which separate $F$.
\end{definition}
The concept of hereditarily non uniformly perfect was introduced in
\cite{SSS} and can be thought of as a thinness criterion for sets which is a strong version of failing to be uniformly perfect.
\begin{definition}
A compact set $E$ is called \textit{hereditarily non uniformly perfect} (HNUP) if no subset of $E$ is uniformly perfect.
\end{definition}

In our case we will show that the iterated Julia sets for suitably chosen polynomial sequences are HNUP by showing they satisfy the stronger property of \emph{pointwise thinness}. A set $E \subset \C$ is called \emph{pointwise thin}
when for each $z \in E$ there exist $0<r_n<R_n$ with $R_n/r_n \to +\infty$ and $R_n \to 0$ such that each true annulus $\mathrm{A}(z ,r_n, R_n)$ separates $E$. A conformal annulus of large modulus which separates a set $E$ contains a round annulus of large modulus (see, e.g., Theorem 2.1 of~\cite{McMullen1}) which then also separates $E$.  We thus have an equivalent formulation (that we shall use later), namely that $E$ is pointwise thin if, for each $z \in E$, there exists a sequence of conformal annuli $A_n$ each of which separates $E$, has $z$ in the bounded component of its complement, and such that $\textrm{mod }A_n \to +\infty$ while the Euclidean diameter of $A_n$ tends to zero.

Note that any pointwise thin compact set is HNUP. Stankewitz, Sugawa, and Sumi used pointwise thinness to establish the HNUP property for several examples in their paper \cite{SSS}. However, they also pointed out that this property is stronger than HNUP and gave an example, originally due to Curt McMullen in~\cite{McMullen2}, of a set of positive $2$-dimensional Lebesgue measure which is HNUP but not pointwise thin.

\subsection{Statements of the Main Results}

The construction of the sequences of polynomials we consider in this paper begins with a sequence $\{c_k\}_{k=1}^\infty$ in $\C^\N$ where we require that $|c_k| > 4$ for every $k$. Using this, we define a sequence $\{m_k\}_{k=1}^\infty$ of natural numbers for which we have, for each $k \ge 1$,

\begin{equation}\label{Invariance*} 2^{2^{m_k}} \ge 2 \sqrt{|c_k|}.
\end{equation}

Since $|c_k| > 4$ for every $k$, we clearly then have the following weaker condition which will suffice for most of our results
\begin{equation}\label{Invariance1}2^{2^{m_k}} > \sqrt{|c_k|} + 1.
\end{equation}

 \newpage

\afterpage{\clearpage}

\begin{figure}[htbp]
  \centering
  \vspace{0.2cm}
\includegraphics[scale=0.75]{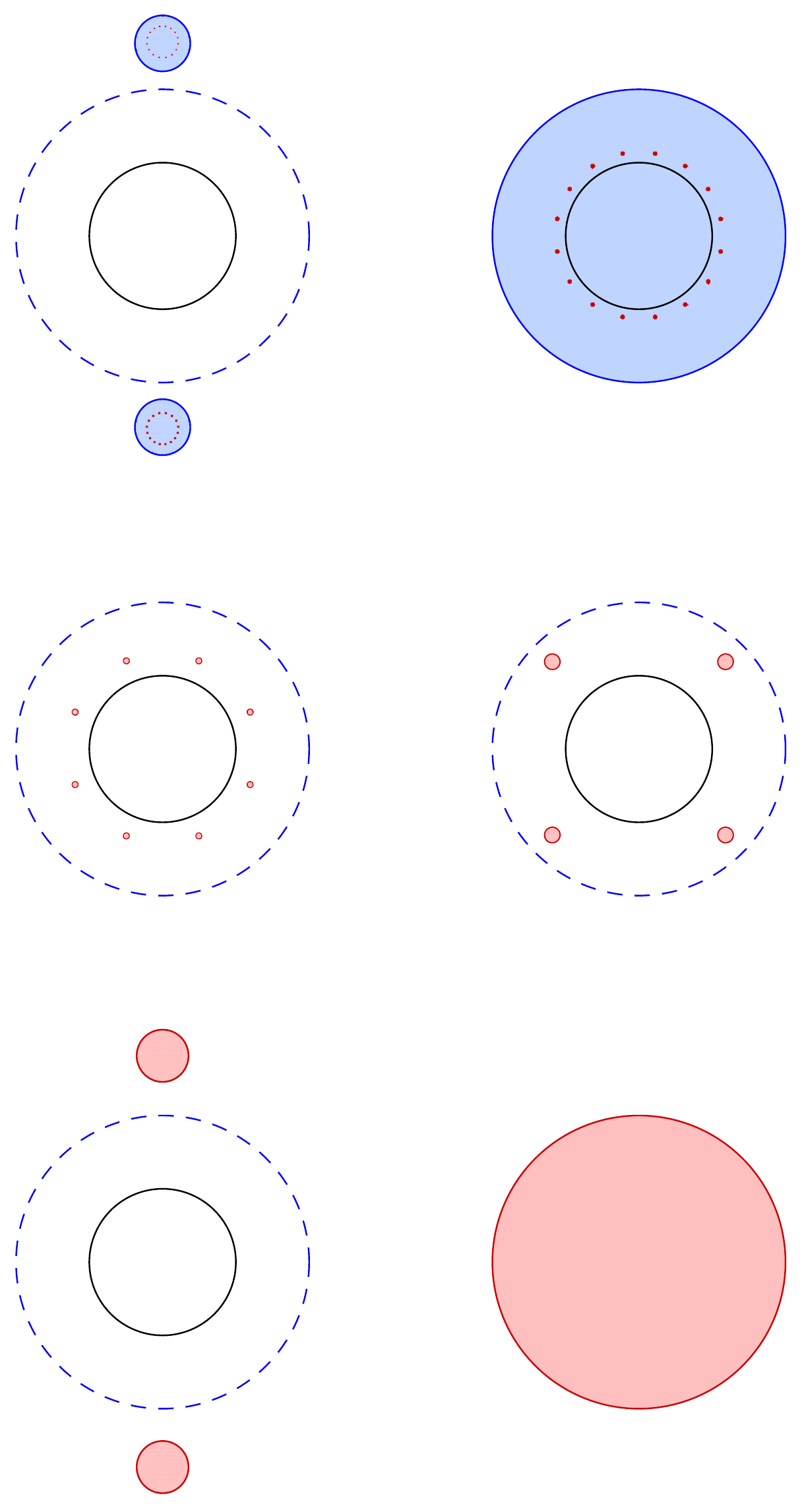}

  \vspace{1cm}
  \caption{How the survival sets ${\mathcal S}_k$ are nested. The pictures show preimages of $\overline {\mathrm D}(0,2)$ at stages $M_k$ (in red) and $M_{k-1}$ (in blue) with $m_k = 3$. The dashed blue circle is ${\mathrm C}(0,2)$ while the unit circle is shown in black.
      Observe how $Q_{M_{k-1},M_k}^{-1}(\overline {\mathrm D}(0,2)) \subset \overline {\mathrm D}(0, 2) \setminus \overline {\mathrm D}(0, 1)$ as in Remark~\ref{Remark1}(\ref{Remark1c}) is shown in red in Stage $M_{k-1}$.}\label{Preimages}
    \unitlength1cm
\begin{picture}(0.01,0.01)
  \put(2,4.4){\footnotesize Stage $M_k$  }
    \put(-5.1,4.4){\footnotesize Stage $M_{k-1}+3 = M_{k-1}+m_k$  }
    \put(1.5,11.3){\footnotesize Stage $M_{k-1}+2$}
    \put(-4.0,11.3){\footnotesize Stage $M_{k-1}+1$}
    \put(1.9,16.4){\footnotesize Stage $M_{k-1}$  }
    \put(-5.2,16.4){\footnotesize Stage $M_{k-1}-1 = M_{k-2}+m_{k-1}$  }
    \put(-.7,7.65){\vector(1, 0){1.4}}
    \put(-.32,7.9){\footnotesize $P_{M_{k}}$}Ä
        \put(-.65,13.5){\vector(1, 0){1.3}}
    \put(-.5,13.7){\footnotesize $z \mapsto z^2$}
        \put(-.7,19.6){\vector(1, 0){1.3}}
    \put(-.45,19.9){\footnotesize $P_{M_{k-1}}$}
            \put(1.3,17.7){\vector(-1, -1){2.6}}
                        \put(1.3,11.85){\vector(-1, -1){2.6}}
 \end{picture}
\end{figure}
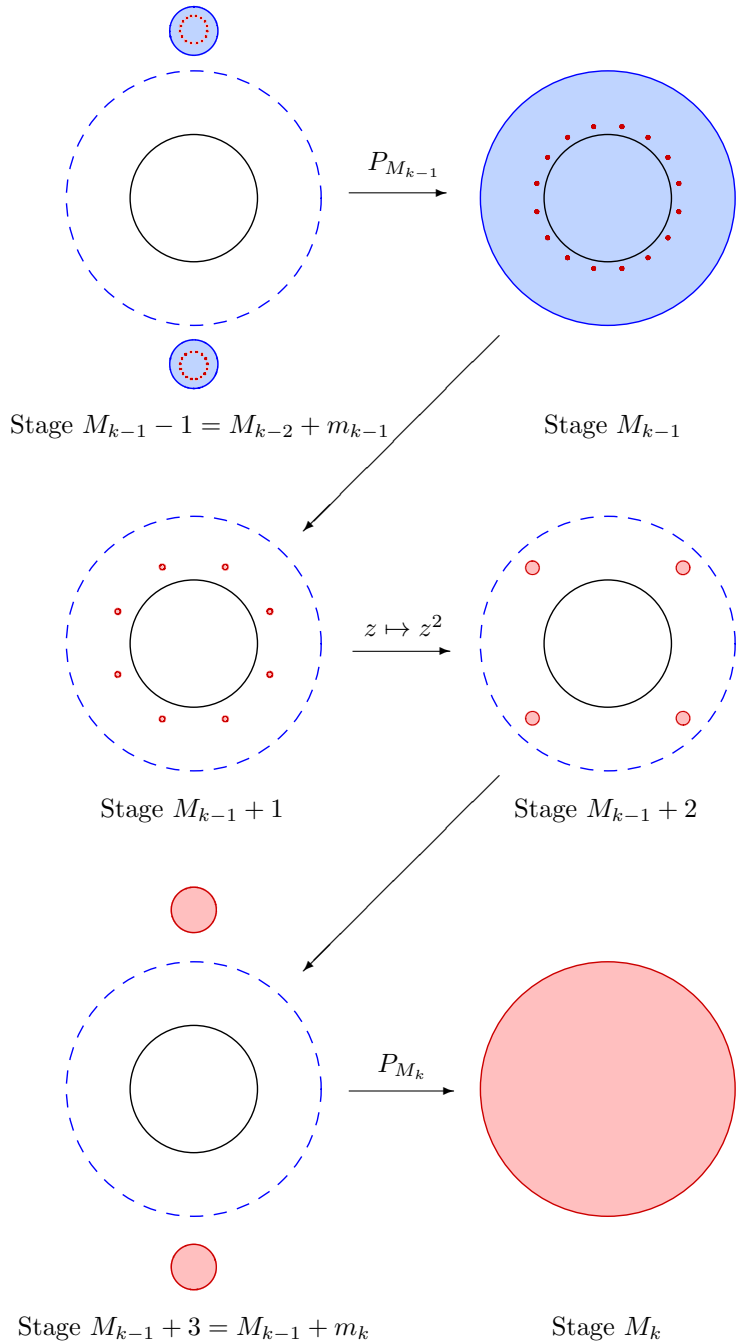

\newpage

Now set $M_0 = 0$, $M_k = \sum_{j=1}^k {(m_j + 1)}$ for each $k \ge 1$, and define a sequence of quadratic polynomials $\Pm$ by
\vspace{0.2cm}
\[P_m = \left \{ \begin{array}{r@{\: ,\quad}l}
z^2 + c_k & \mbox{if} \:\; m=M_k \quad \mbox{for some} \:\;k \,\ge 1
\vspace{.2cm}\\

z^2 &  \mbox{otherwise.} \\
\end{array} \right .  \]

Hence each $Q_{M_{k-1},M_{k-1}+m_k}(z) = z^{2^{m_k}}$, $Q_{M_{k-1},M_{k}}(z) = z^{2^{m_k+1}}+c_k$, $Q_{M_k}= Q_{M_{k-1},M_{k}}\circ \dots \circ Q_{M_1,M_{2}}\circ Q_{M_0, M_{1}}$ has degree $2^{M_k}$, and we have the following three observations.

\begin{remark} \label{Remark1}

\begin{enumerate}[(a)]
\item\label{Remark1a} Note that~\eqref{Invariance*} ensures that the image of the closed disc $\overline {\mathrm D}(0, 2)$ under $m_k$ iterations of $z^2$, i.e., $Q_{M_{k-1},M_{k-1}+m_k}$, will cover the disc $\overline {\mathrm D}(0,2\sqrt{|c_k|}) \supset \overline {\mathrm D}(0,\sqrt{|c_k|}+ 1)$.  However, for all but the proof of Theorem~\ref{ThmDichotomy}, we only require the consequence of the weaker inequality~\eqref{Invariance1} that gives that the image of the closed disc $\overline {\mathrm D}(0, 2)$ under  $Q_{M_{k-1},M_{k-1}+m_k}$ will cover the disc $\overline {\mathrm D}(0,\sqrt{|c_k|}+ 1)  \supset \overline {\mathrm D}(\pm \sqrt{-c_k}, 1)$.

    Since $|c_k|>4$, we see that $\overline {\mathrm D}(\sqrt{-c_k}, 1)$ lies outside $\overline {\mathrm D}(0, 1)$, and so it follows that $Q_{M_{k-1},M_{k-1}+m_k}^{-1}(\overline {\mathrm D}(\sqrt{-c_k}, 1))$ consists of $2^{m_k}$ components, each of which is contained in $\overline {\mathrm D}(0, 2) \setminus \overline {\mathrm D}(0, 1)$.  Similarly, by also considering $Q_{M_{k-1},M_{k-1}+m_k}^{-1}(\overline {\mathrm D}(-\sqrt{-c_k}, 1))$, we note that one can quickly conclude that the preimage under $Q_{M_{k-1},M_{k-1}+m_k}$ of $\overline {\mathrm D}(\sqrt{-c_k}, 1) \cup \overline {\mathrm D}(-\sqrt{-c_k}, 1)$ consists of $2^{m_k+1}$ components, each differing from another by a rotation (about $0$) by a multiple of $\tfrac{2\pi}{2^{m_k + 1}}.$

\item \label{Remark1b} We also note that the preimage of $\overline {\mathrm D}(0, 2)$ under $P_{M_{k}}(z) = z^2 + c_{k}$ with $|c_{k}| > 4$ consists of two components about $\pm \sqrt{-c_{k}}$  which are contained in the two discs $\overline {\mathrm D}(\sqrt{-c_{k}}, 1)$, $\overline {\mathrm D}(-\sqrt{-c_{k}}, 1)$.  This is quickly seen by noting that the derivative of the inverse branches of $P_{M_{k}}$ have modulus less than $1/2$ on $\overline {\mathrm D}(0, 2)$.

\item \label{Remark1c}  Since $|c_k|>4$, the map $Q_{M_{k-1},M_k}$ has a full set of $2^{m_k+1}$ inverse branches on ${\mathrm D}(0,4)$.  Hence, by parts~(\ref{Remark1a}) and~(\ref{Remark1b}), we see that the set $Q_{M_{k-1},M_k}^{-1}(\overline {\mathrm D}(0,2))=Q_{M_{k-1},M_{k-1}+m_k}^{-1}(P_{M_{k}}^{-1}(\overline {\mathrm D}(0,2)))$ is shown to consist of $2^{m_k+1}$ components, one for each branch, each of which is contained in $\overline {\mathrm D}(0, 2) \setminus \overline {\mathrm D}(0, 1)$.  (See Figure~\ref{Preimages}.)
\end{enumerate}
\end{remark}

Given such a sequence, for each $k \ge 1$, we define the \emph{$k$th survival set ${\mathcal S}_k$ at time $0$} by
\begin{equation}\label{Sk}
{\mathcal S}_k = Q_{M_k}^{-1}(\overline {\mathrm D}(0, 2)) = Q_{M_0, M_{1}}^{-1} ( \cdots \cdots (Q_{M_{k-1},M_{k}}^{-1}(\overline {\mathrm D}(0, 2)))\cdots).
\end{equation}

 \newpage

\begin{remark} \label{Invariance2} Note that the sets ${\mathcal S}_k \subset \overline {\mathrm D}(0, 2) \setminus \overline {\mathrm D}(0, 1)$ are decreasing in $k$ by Remark~\ref{Remark1}(\ref{Remark1c}).  Also, $Q_{M_k}({\mathcal S}_{k+1}) = Q_{M_{k},M_{k+1}}^{-1}(\overline {\mathrm D}(0, 2))\subset \overline {\mathrm D}(0, 2) \setminus \overline {\mathrm D}(0, 1)$.  Lastly, by repeatedly applying Remark~\ref{Remark1}(\ref{Remark1c}), ${\mathcal S}_k$ consists of $2^{M_k}$ components.
\end{remark}
    
Given this, our first theorem is as follows:

\begin{theorem}\label{ThmJm}
For a sequence $\Pm$ as above, we have $\mathcal{J}=\bigcap_{k \ge 1} {\mathcal S}_k$.  Consequently, for each $m \ge 0$,

\[{\mathcal J_m} = Q_m \left( \bigcap_{k \ge 1} {\mathcal S}_k \right ) =  \bigcap_{k \ge 1} Q_m({\mathcal S}_k).\]
\end{theorem}

Using this, we are able to prove the main result of our paper:

\begin{theorem}\label{ThmDichotomy}
For a sequence $\Pm$ as above, $\Jm$ is uniformly perfect for every $m \ge 0$ if and only if $\{c_k\}_{k=1}^\infty$ is bounded, and $\Jm$ is pointwise thin and HNUP  for every $m \ge 0$ if and only if $\{c_k\}_{k=1}^\infty$ is unbounded.
\end{theorem}

We have the following three important observations.

\begin{remark} \label{Remark3}

\begin{enumerate}[(a)]
\item We note that the existence of a HNUP Julia set is a new phenomenon related to non-autonomous dynamics of unbounded sequences that is not present in classical rational iteration or (non-elementary) semigroup dynamics. In particular, the Julia set of a rational function of degree two or more is uniformly
perfect (see~\cite{Er, Hi, MdR}). 
Also, the Julia set of a bounded sequence of polynomials is uniformly perfect (see Theorem 1.6 of~\cite{Sumi3}]). 

\item Furthermore, by~\cite{RS3}, the Julia set of any non-elementary rational semigroup $G$, which is allowed to contain or even consist of M\"obius maps, is uniformly perfect when there is a uniform upper bound on the Lipschitz constants (with respect to the spherical metric) of the generators of $G$.  Hence we justify our claim in (a) above as follows.    Suppose that $G$ is a non-elementary rational semigroup (i.e., its Julia set $J(G)$ is such that $\# J(G) \geq 3$), with no assumption regarding the Lipschitz constants of the generators.  Since the repelling fixed points of the elements of $G$ are then dense in $J(G)$ (see~\cite{StankewitzRepelDense2}), we may select distinct $a, b, c \in J(G)$ to be repelling fixed points of maps $f, g, h$ in $G$.  Denoting by $G'=\langle f, g, h \rangle$, the subsemigroup of $G$ generated by $f, g, h$, we then must have that $J(G)$ contains the uniformly perfect set $J(G')$, and hence $J(G)$ is not HNUP. 

\item If $J_{1}\subset \C $ and $J_{2}\subset \C$ are 
topological Cantor sets, $J_{1}$ is uniformly perfect, and $\varphi : \chat
\setminus J_{1} \rightarrow \chat \setminus J_{2}$ is a 
quasiconformal homeomorphism, then $J_{2}$ is also uniformly perfect.
Thus, if $J_{2}$ is a HNUP iterated Julia set (at some time $m \ge 0$) of some polynomial sequence (e.g. as in 
Theorem \ref{ThmDichotomy}) and
$J_{1}$ is a uniformly perfect iterated Julia set of some polynomial sequence
(e.g. the Julia set of iteration of a single polynomial of degree two or 
more), then there exists no quasiconformal map
$\varphi : \chat \setminus J_{1}\rightarrow \chat 
\setminus J_{2}.$ In particular, this implies that none of the sequences in Theorem \ref{ThmDichotomy} with HNUP iterated Julia sets can be conjugate via quasiconformal mappings to a sequence whose Julia sets are uniformly perfect Cantor sets.

\end{enumerate}
\end{remark}

Since our sets $\Jm$ are basically fractal constructions, it is of interest to know as much as possible about their Hausdorff dimensions $\HD (\Jm)$.

\begin{theorem}\label{ThmHDbound}
For any sequence  $\Pm$ as above, for each $m \ge 0$, $\HD (\Jm) \le 1$.
\end{theorem}

On the other hand, hereditarily non uniformly perfect is a notion of thinness of sets and it is therefore interesting to find examples of HNUP sets which nevertheless have positive Hausdorff dimension as was done by Stankewitz, Sugawa, and Sumi in \cite{SSS}. This is also the case with our examples, and the upper bound given in the statement of the above result can, in fact, be attained.

\begin{theorem}\label{ThmHDmax}
There exists a sequence $\Pm$ as above such that, for each $m \ge 0$, $\Jm$ is pointwise thin and HNUP but $\HD (\Jm) = 1$.\end{theorem}

The organization of the remainder of this paper is as follows.  In Section~\ref{SectProofThm1.2and1.3}, we state and prove some ancillary lemmas and give the proofs of Theorems~\ref{ThmJm} and~\ref{ThmDichotomy}. Roughly speaking, Theorem~\ref{ThmJm} says that the Julia set $\mathcal{J}$ is the limit set of a suitable non-autonomous conformal iterated function system, as considered in the paper of Rempe-Gillen and Urb\'{a}nski~\cite{RU}. This is the point of view we will adopt in Section~\ref{SectHD} when we turn to considering the Hausdorff dimensions of the iterated Julia sets.  In particular, we use it to prove Theorem~\ref{ThmHDbound}, and then, using Bowen's formula given in~\cite{RU} (restated here as Theorem~\ref{BowenFormula}), we show that we can choose our sequence of constants $\{c_k\}_{k=1}^\infty$ and integers $\{m_k\}_{k=1}^\infty$ to prove Theorem~\ref{ThmHDmax}, that is, to obtain the highest possible Hausdorff dimension.

\section{Proofs of Theorems~\ref{ThmJm} and~\ref{ThmDichotomy}}\label{SectProofThm1.2and1.3}

We first prove two small lemmas which will be of use to us in obtaining Theorems~\ref{ThmJm} and~\ref{ThmDichotomy} on characterizing the iterated Julia sets and obtaining HNUP examples (respectively).

\begin{lemma}\label{Lem2.1} For a sequence $\Pm$ as above, we have the following.
\begin{enumerate}[(a)]

\item \label{Lemma2.1b} For any $k \ge 1$ and any $n \ge 2$, if $|z| > 2^n$, then

\[ |Q_{M_{k-1},M_{k}}(z)| > 2^{n+1}.\]

\item \label{Lemma2.1c} For each $z_0 \in \bigcap_{k \ge 1} {\mathcal S}_k$, the orbit $\{Q_m(z_0)\}_{m=0}^\infty$ lies entirely outside of the closed unit disk.

\end{enumerate}
\end{lemma}

\begin{proof}
For part~(\ref{Lemma2.1b}), we first note that $f(x, y)=2^{xy}-2^y-2^{x+1} \ge f(2, 4)>0$ for all $(x, y) \in R := [2, +\infty) \times [4, +\infty)$, which is an immediate consequence of the mean value theorem and the fact that the partial derivatives $f_x$ and $f_y$ are each strictly positive on $R$.

Now let $|z| > 2^n$ where $n \geq 2$.  Since $|c_{k}| <  2^{2^{m_{k}+1}}$ by inequality~\eqref{Invariance1}, applying the above using $x=n$ and $y=2^{m_{k}+1}$ gives
$|Q_{M_{k-1},M_{k}}(z)| = |z^{2^{m_{k}+1}}+c_{k}|> (2^n)^{2^{m_{k}+1}}  -  2^{2^{m_{k}+1}} > 2^{n+1}.$

We let $z_0 \in \bigcap_{k \ge 1} {\mathcal S}_k$, and prove part~(\ref{Lemma2.1c}) by contradiction.  Suppose not, and call $m_0$ the smallest index such that $Q_{m_0}(z_0) \in \overline {\mathrm D}(0,1)$.  If $m_0$ is not equal to any $M_k$ (note that, since $M_0 = 0$, in particular this implies that $m_0 \ge 1$), then $P_{m_0}(z)=z^2$ and we have a contradiction (to the minimality of $m_0$) since that would imply $Q_{m_0-1}(z_0) \in \overline {\mathrm D}(0,1)$ (else we could not have $P_{m_0}(Q_{m_0-1}(z_0))=Q_{m_0}(z_0) \in \overline {\mathrm D}(0,1)$).  However, if $m_0=M_{k_0}$ for some $k_0 \ge 0$, then we see that, since $z_0 \in  {\mathcal S}_{k_0+1}$, Remark~\ref{Invariance2} gives that $Q_{M_{k_0}}({\mathcal S}_{k_0+1}) = Q_{M_{k_0},M_{k_0+1}}^{-1}(\overline {\mathrm D}(0, 2))\subset \overline {\mathrm D}(0, 2) \setminus \overline {\mathrm D}(0, 1)$, which yields a contradiction.
\end{proof}

\begin{lemma}\label{Lem2.2}
Let $\Pm$ be a sequence of quadratic polynomials as above:

\begin{enumerate}[(a)]
\item \label{Lem2.2a} For $0 \le m < n$ and $z \in Q_m  \left (\bigcap_{k \ge 1} {\mathcal S}_k \right )$, we have $|Q'_{m,n}(z)| \ge 2^{n-m}$.

\item \label{Lem2.2b} Let $k \ge 1$ and let $f(z) = (z-c_k)^{1/2^{m_k+1}}$ be any inverse branch of $Q_{M_{k-1},M_k}$, which is defined on ${\mathrm D}(0,4)$.  Then, for $0 \leq \epsilon \leq 1$, we have $$\sup\{|f'(z)|:z\in \overline{{\mathrm D}}(0,2+\epsilon)\}=\frac{1}{2^{m_k+1}} (|c_k|-2-\epsilon)^{\left(\frac{1}{2^{m_k+1}}-1\right)} \leq \eta_\epsilon,$$ where $\eta_\epsilon :=\frac{1}{2^{1+1}} (2-\epsilon)^{\left(\frac{1}{2^{1+1}}-1\right)}$.  In particular, using $\epsilon =0$ gives
    $$\sup\{|f'(z)|:z\in \overline{{\mathrm D}}(0,2)\}=\frac{1}{2^{m_k+1}} (|c_k|-2)^{\left(\frac{1}{2^{m_k+1}}-1\right)} \leq \eta=2^{-\frac{11}{4}}<1.$$
\end{enumerate}
\end{lemma}

\begin{proof}
Part~(\ref{Lem2.2a}) follows immediately by Lemma~\ref{Lem2.1}(\ref{Lemma2.1c}) and the fact that the absolute value of the derivative of any quadratic of the form $z^2 + c$ is greater than $2$ at any point outside the closed unit disk.

To prove part~(\ref{Lem2.2b}), note that $|f'(z)|=\frac{1}{2^{m_k+1}} |z-c_k|^{\left(\frac{1}{2^{m_k+1}}-1\right)}$.  Since $|c_k|>4$ and $m_{k} \geq 1$, we then have that
\begin{eqnarray*}
\sup\{|f'(z)|:z \in {\mathrm D}(0,2+\epsilon) \}&=&\frac{1}{2^{m_k+1}} (|c_k|-2-\epsilon)^{\left(\frac{1}{2^{m_k+1}}-1\right)}\\
 &\leq& \frac{1}{2^{m_k+1}} (2-\epsilon)^{\left(\frac{1}{2^{m_k+1}}-1\right)}\\
 &\leq& \frac{1}{2^{1+1}} (2-\epsilon)^{\left(\frac{1}{2^{1+1}}-1\right)}=\eta_\epsilon.  \qedhere
\end{eqnarray*}
\end{proof}

\begin{proof}[Proof of Theorem~\ref{ThmJm}]
We first prove the result for $m = 0$, basing our proof on showing that $\bigcap_{k \ge 1} {\mathcal S}_k$ is
precisely the set of points whose orbits do not escape locally uniformly to infinity.

Suppose first that $z \notin \bigcap_{k \ge 1} {\mathcal S}_k$, i.e., $|Q_{M_k}(z)| > 2$ for some $k$. From~\eqref{Invariance1} we then get that $|Q_{{M_k}+m_{k+1}}(z)| > \sqrt{|c_{k+1}|} + 1$ and so, since $|c_{k+1}| > 4$, we obtain $|Q_{M_{k+1}}(z)| > 5 > 4$. It then follows easily from Lemma~\ref{Lem2.1}(\ref{Lemma2.1b}) that $Q_{M_j}(z)\to \infty$ as $j \to \infty$.  Note that, for each $j \ge k$ and $0 \leq N  \le m_{j+1}$, since $Q_{M_j,M_j+N}(z)=z^{2^N}$, we see that $|Q_{M_j+N}(z)|=|Q_{M_j,M_j+N}(Q_{M_j}(z))|>|Q_{M_j}(z)|$.  From this it clearly follows that $Q_{m}(z)\to \infty$ as $m \to \infty$, and at a rate which is locally uniform, whence we must have that $z \in {\mathcal F}$.

On the other hand, let $z \in  \bigcap_{k \ge 1} {\mathcal S}_k$.  Then $|Q_{M_k}(z)| \le 2$ for every $k$, while  Lemma~\ref{Lem2.2}(\ref{Lem2.2a}) yields that $|Q_{M_k}'(z)| > 2^{M_k} \to \infty$ as $k \to \infty$.  This shows that no subsequence of $\{Q_{M_k}\}$ can converge locally uniformly (to what would have to be a holomorphic function) in any neighbourhood of $z$, whence $z \in \J$ as desired.

The result for all $m \ge 0$ then follows immediately from complete invariance (Theorem~\ref{ThmCompInvar}) and the fact that the sets ${\mathcal S}_k$ are nested and compact.
\end{proof}

\newpage

\afterpage{\clearpage}

\begin{sidewaysfigure}[htb]
  \centering
  \vspace{15cm}
  \includegraphics[width=20cm]{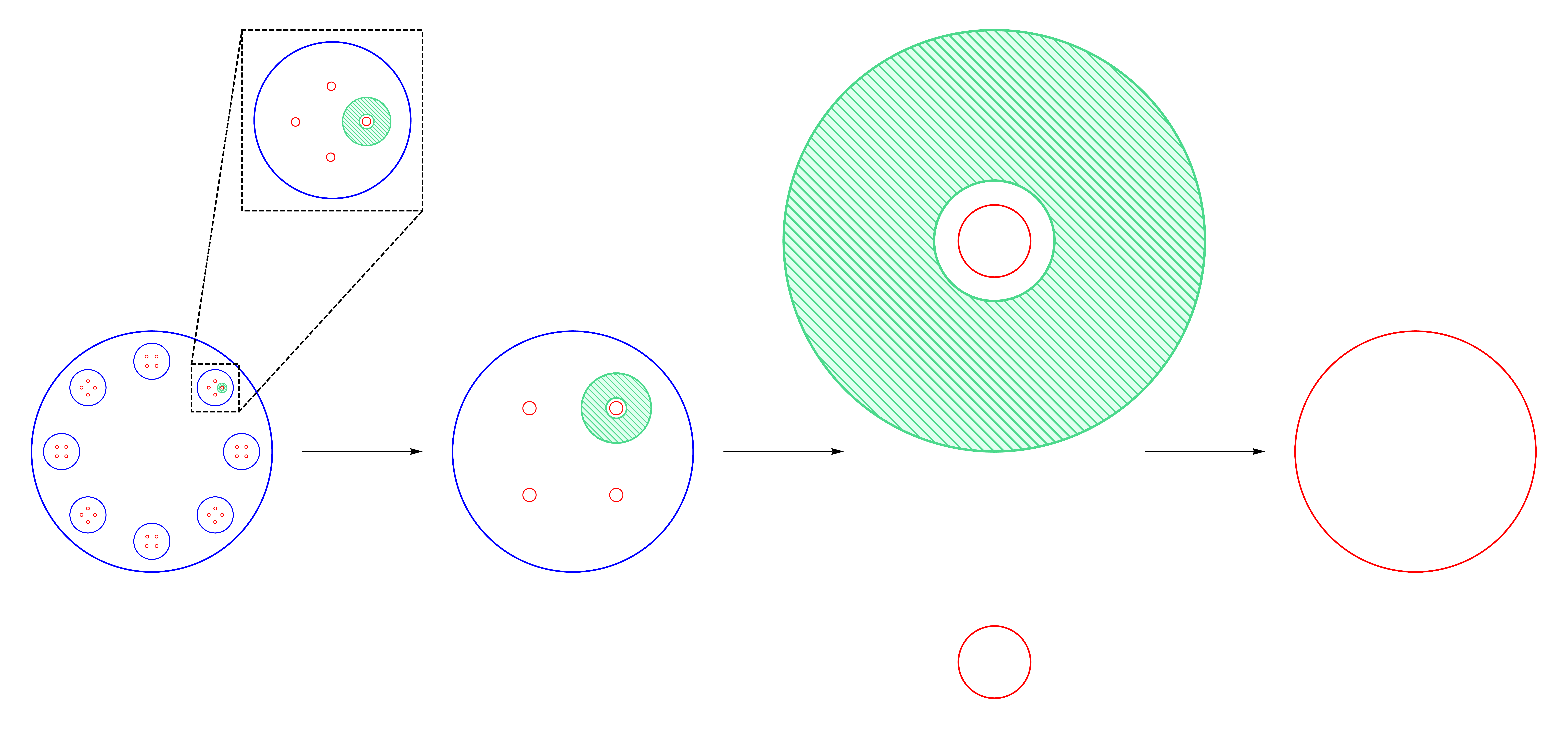}

  \vspace{.9cm}
  \caption{Schematic for the proof of Theorem 1.6 in the case where $\limsup |c_k| = +\infty$. Note how the round annulus ${\mathrm A}(\sqrt{-c_{k}}, 1, \sqrt{|c_{k}|})$ at stage $M_{k-1} + m_k$ (in this case $M_1+m_2$) is pulled back conformally first by the preimage branches of $Q_{M_{k-1},M_{k-1}+m_k}$ to form half the members of the collection $\mathcal C$ at Stage $M_1$.  Then the preimage branches of $Q_{M_{k-1}}$ pull back the annuli in $\mathcal C$ (one of which is visible in the zoomed box) to conformal annuli which separate the components of $\mathcal{S}_k$ at stage $0$.}  \label{MappingsPic}
    \unitlength1cm
\begin{picture}(0.01,0.01)
\put(-9,3.6){\footnotesize Stage $0$}
\put(-3.5,3.6){\footnotesize Stage $M_1$}
  \put(1.5,3.6){\footnotesize Stage $M_1+m_2$  }
    \put(7.1,3.6){\footnotesize Stage $M_2$}
    \put(-5.7,7.3){\footnotesize $Q_{M_1}$}
     \put(-.9,7.3){\footnotesize $Q_{M_1,M_1+m_2}$}
      \put(5.1,7.3){\footnotesize $P_{M_2}$}
       \put(-3.3,5.5){\footnotesize \color{blue}${\mathrm C}(0,2)$}
     \put(7.45,5.5){\footnotesize \color{red}${\mathrm C}(0,2)$}
     \put(1.3,13.3){\footnotesize \textcolor[rgb]{0.1,0.6,0.33}{${\mathrm A}(\sqrt{-c_{2}}, 1, \sqrt{|c_{2}|})$}}
     \put(-7,13.3){\footnotesize Zoom of Stage $0$}
 \end{picture}
\end{sidewaysfigure}
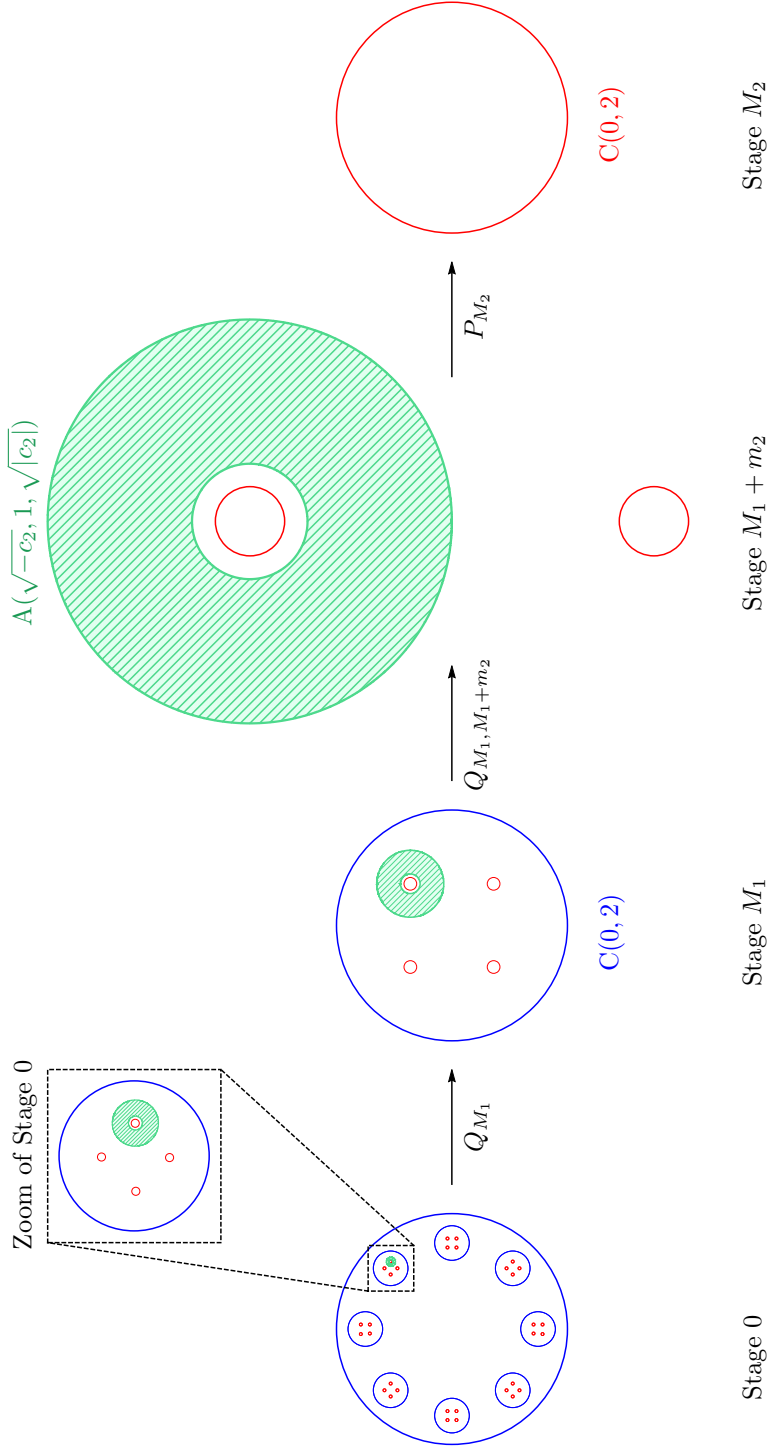

\newpage

\begin{remark}\label{RemarkStrongInequality}
The proofs presented for the previous results together with the results on Hausdorff dimension proved in Section~\ref{SectHD} only require the weaker inequality in~\eqref{Invariance1}.  Only in the next proof of the HNUP property do we employ the stronger inequality in~\eqref{Invariance*}.
\end{remark}

\begin{proof}[Proof of Theorem~\ref{ThmDichotomy}]
If the sequence $\{c_k\}_{k=1}^\infty$ is bounded, then the polynomial sequence $\Pm$ is bounded and it is well known that the iterated Julia sets for a bounded polynomial sequence are uniformly perfect.
Moreover, if $\{ P_{m}\}_{m=1}^{\infty}$ is a sequence of rational maps such that $\deg (P_{m})\geq 2$ for each
$m\in \Bbb{N}$ and such that $\{ P_{m}\mid m\in \Bbb{N}\} $ is relatively compact in the space Rat of all rational maps endowed with 
the topology of uniform convergence on the Riemann sphere, then the Julia set of the sequence
$\{ P_{m}\}_{m \in \mathbb{N}}$ is uniformly perfect. These results follow from Theorem 1.26 of~\cite{Sumi3}
(where one considers the skew product on the Riemann sphere on the closure of
$\{ \sigma ^{n}(P_{1},P_{2},\ldots )\mid n\in \mathbb{N}\cup \{ 0\} \}$
where $\sigma :\mbox{Rat}^{\Bbb{N}}\rightarrow \mbox{Rat}^{\Bbb{N}}$ denotes the
shift map on the infinite product space $\textrm{Rat}^{\mathbb{N}}$,
which is a compact metric space).

Now suppose $\limsup |c_k| = +\infty$. We first show that $\J$ is HNUP by showing it is pointwise thin as defined in Section~\ref{SectHNUP} via the formulation in terms of conformal annuli.

Fix $k \ge 1$.  As noted in Remark~\ref{Remark1}(\ref{Remark1b}) and illustrated in Figure~\ref{Preimages}, $P_{M_{k}}^{-1}(\overline {\mathrm D}(0, 2))$ consists of two components about $\pm \sqrt{-c_{k}}$ which are contained in the two discs $\overline {\mathrm D}(\sqrt{-c_{k}}, 1)$, $\overline {\mathrm D}(-\sqrt{-c_{k}}, 1)$. Hence the (round) annulus ${\mathrm A}(\sqrt{-c_{k}}, 1, \sqrt{|c_{k}|})$ separates $P_{M_{k}}^{-1}(\overline {\mathrm D}(0, 2))$.  Consider an open slit plane $S = \C \setminus R$, where $R$ is a ray emanating from the origin which does not meet either of the open disks ${\mathrm D}(\sqrt{-c_{k}}, \sqrt{|c_{k}|})$, ${\mathrm D}(-\sqrt{-c_{k}}, \sqrt{|c_{k}|})$.  (For example, in the case that $c_k>0$, $R$ could be either $(-\infty, 0]$ or $[0, +\infty)$ as illustrated in Figure~\ref{MappingsPic} where $k=2.$)

Since $S$ is simply connected and does not contain the origin, the map $Q_{M_{k-1},M_{k-1}+m_k}(z) = z^{2^{m_k}}$ has $2^{m_k}$ inverse branches defined on $S$, with each differing by a factor of a $2^{m_k}$-th root of unity.  Call one such inverse branch $f$, and note $A_k:=f({\mathrm A}(\sqrt{-c_{k}}, 1, \sqrt{|c_{k}|})) \subset f(\mathrm D(\sqrt{-c_{k}},\sqrt{|c_{k}|}))$ is a conformal annulus of modulus $\frac{\log |c_k|}2$ which, by~\eqref{Invariance*} (see also  Remark~\ref{Remark1}(\ref{Remark1a})), lies entirely in $\mathrm D(0,2)$ and separates one of the $2^{m_k+1}$ components of $Q_{M_{k-1},M_k}^{-1}(\overline {\mathrm D}(0,2))=Q_{M_{k-1},M_{k-1}+m_k}^{-1}(P_{M_{k}}^{-1}(\overline {\mathrm D}(0,2)))$ from each of the other components.  Clearly, by rotational symmetry about the origin, we can obtain a collection $\mathcal{C}$ of $2^{m_k+1}$ such conformal annuli, each separating a different one of the $2^{m_k+1}$ components of the set $Q_{M_{k-1},M_k}^{-1}(\overline {\mathrm D}(0,2))$ from each of the other components.

Now note that, by applying Remark~\ref{Remark1}(\ref{Remark1c}) repeatedly, $Q_{M_{k-1}}$ has all $2^{M_{k-1}}$ of its inverse branches defined and univalent on a neighbourhood of $\overline {\mathrm D}(0,2))$ (for another perspective we will see later in Section~\ref{SectHD}, these are just the maps of the form $\phi_\omega=\phi_{\omega_1}^{(1)} \circ \dots \circ \phi_{\omega_{k-1}}^{(k-1)}$ for all $\omega= \omega_1 \dots \omega_{k-1} \in I^{k-1}$).  Applying each such inverse branch to each annulus in $\mathcal{C}$ generates a collection of $2^{M_{k-1}}\cdot 2^{m_k+1}=2^{M_k}$ conformal annuli each having modulus $\frac{\log |c_k|}2$, separating one of the $2^{M_k}$ components of $Q_{M_{k}}^{-1}(\overline {\mathrm D}(0,2)))=\S_{k}$ from all other such components, and lying entirely in a component of $Q_{M_{k-1}}^{-1}(\overline {\mathrm D}(0,2))=\S_{k-1}$.  (Here, of course, we trivially set $\S_0 =\overline {\mathrm D}(0,2))$ to deal with the notation for the case $k=1$.)

Pick arbitrary $z \in \J$.  By the previous result, $z$ must lie in the bounded component of the complement of a conformal annulus of modulus $\frac{\log |c_k|}2$, which separates $\S_k$ (and therefore separates $\J$ since every component of $\S_k$ clearly contains a point of $\J$) and lies in a component of $\S_{k-1}$.  Lemma~\ref{Lem2.2}(\ref{Lem2.2b}), applied repeatedly, shows that each component of $\S_{k-1}$ has diameter no larger than $4\cdot \eta^{k-1}$, and so must shrink to zero as $k \to \infty$.  Since $\limsup \frac{\log |c_k|}2=+\infty$, we must have pointwise thinness of $\J$.

To extend this result to all the iterated Julia sets $\Jm$, we first observe that if we fix $k \ge 1$ and consider the truncated sequences $\{m_j\}_{j=k}^\infty$, $\{c_j\}_{j=k}^\infty$, then the corresponding polynomial sequence $\{P_m\}_{m=M_k +1}^\infty$ still trivially satisfies the same lower bound on the absolute values of the constants $c_k$ and
the same invariance condition \eqref{Invariance*}. This allows us to conclude that the Julia set at time $0$ for this truncated sequence, which is the same as $\J_{M_k}$ (the iterated Julia set at time $M_k$ for our \emph{original sequence}
$\Pm$), satisfies the pointwise thinness property where we again know that our separating annuli in our collection ${\mathcal C}_{M_k}$ of arbitrarily large modulus as above lie inside $\overline{\mathrm D}(0, 2)$. Now pick $m \ge 0$ arbitrary and not equal to any $M_j$.  Then choose $k$ as small as possible so that $m < M_k$.  The composition $Q_{m,M_k}(z) = z^{2^{M_k -m}}+c_k$ has a single critical value $c_k$ which avoids $\overline{\mathrm D}(0, 2)$. By Theorem~\ref{ThmCompInvar}, $Q_{m,M_k}^{-1}(\J_{M_k}) = \Jm$. The desired conclusion for $\Jm$ then follows on taking the preimages under $Q_{m,M_k}$ of the conformal annuli  in ${\mathcal C}_{M_k}$ which separate $\J_{M_k}$.
\end{proof}

\begin{remark}
The pointwise thinness of $\J_m$ can also be seen to follow from that of $\J$ by using the complete invariance of Theorem~\ref{ThmCompInvar} and noting that pointwise thinness property is preserved under analytic mappings.  We leave the details to the reader.
\end{remark}

\section{Results on Hausdorff Dimension}\label{SectHD}

In order to prove Theorems~\ref{ThmHDbound} and~\ref{ThmHDmax}, we utilize the notion of a non-autonomous conformal iterated function system as presented in \cite{RU} showing, in particular, that $\J$ is the limit set of such a system. The reason we can adopt this approach is that, in our case, the inverse branches of the key maps of our sequence are contractions on a suitable set containing the iterated Julia sets (which follows immediately from Theorem \ref{ThmJm} and part (b) of Lemma \ref{Lem2.2}). 

Here $X$ will always represent a compact subset of $\R^d$ such that $\overline{\textrm{int} X}=X$ with $X$ being such that $\partial X$ is smooth or $X$ is convex (our application below uses $X=\overline{D}(0,2)$ with $\R^d=\R^2=\C$).  Given a conformal map $\phi: X \to X$ we denote by $\phi'(x)$ or $D\phi(x)$ the derivative of $\phi$ evaluated at $x$, i.e., $\phi'(x):\R^d \to \R^d$ is a similarity linear map.  We also put $\|D\phi\|=\|\phi'\|=\sup\{|\phi'(x)|:x \in X\}$, where $|\phi'(x)|$ (or $|D\phi(x)|$) denotes the scaling factor (i.e., matrix norm) of $\phi'(x)$.

\begin{definition} \label{NCIFSdef}
A \emph{non-autonomous conformal iterated function system} (NCIFS) $\Phi$ on the set $X$ is given by a sequence $\Phi^{(1)},\Phi^{(2)},\Phi^{(3)}, \dots$, where each $\Phi^{(j)}$ is a collection of functions $(\phi_i^{(j)}:X \to X)_{i \in I^{(j)}}$ for which $I^{(j)}$ is a finite or countably infinite index set, such that the following hold.

\vspace{-.2cm}
\begin{enumerate}[(A)]
  \item \label{NCIFSa} \emph{Open set condition}:  We have $$\phi^{(j)}_a(\textrm{int}(X)) \cap \phi^{(j)}_b(\textrm{int}(X))= \emptyset$$ for all $j \in \N$ and all distinct indices $a, b \in I^{(j)}$.
  \item \label{NCIFSb} \emph{Conformality}:  There exists an open connected set $V \supset X$ (independent of $i$ and $j$) such that each $\phi_i^{(j)}$ extends to a $C^1$ conformal diffeomorphism of $V$ into $V$.
  \item \label{NCIFSc} \emph{Bounded distortion}: There exists a constant $K \geq 1$ such that for any $k \leq l$ and any $\omega_k, \omega_{k+1}, \dots, \omega_l$ with each $\omega_j \in I^{(j)}$, the map $\phi:=\phi_{\omega_k} \circ \phi_{\omega_{k+1}} \circ \dots \circ \phi_{\omega_l}$ satisfies $$|D\phi(x)| \leq K |D\phi(y)|$$ for all $x, y \in V$.
  \item \label{NCIFSd} \emph{Uniform contraction}:  There is a constant $\eta <1$ such that $$\|D\phi\| \leq \eta^m$$ for all sufficiently large $m$ and all $\phi=\phi_{\omega_j} \circ \dots \circ \phi_{\omega_{j+m-1}}$ where $j \geq 1$ and $\omega_k \in I^{(k)}$.  In particular, this holds if $$\|D\phi_i^{(j)}\| \leq \eta$$ for all $j \geq 1, i \in I^{(j)}$.
\end{enumerate}
\end{definition}

\vspace{.2cm}
\begin{definition}[Words] \label{Wordsdef}
For each $k \in \N$, we define the symbolic space $$I^k:=\prod_{j=1}^k I^{(j)}.$$
\end{definition}
Note that $k$-tuples $(\omega_1, \dots, \omega_k) \in I^k$ may be identified with the corresponding word $\omega_1\dots \omega_k$.

We now give the definition of the limit set of a NCIFS.

\begin{definition} \label{LimitSetdef}
For all $k \in \N$ and $\omega= \omega_1 \dots \omega_k \in I^k$, we define $\phi_\omega=\phi_{\omega_1}^{(1)} \circ \dots \circ \phi_{\omega_k}^{(k)}$ with $$X_\omega:= \phi_\omega (X) \textrm{   and  } X_k:= \bigcup_{\omega \in I^k} X_\omega.$$  The \emph{limit set} (or \emph{attractor}) of $\Phi$ is defined as $$J:=J(\Phi):=\bigcap_{k=1}^\infty X_k.$$
\end{definition}

Note that, in the case where each index set $I^{(j)}$ is finite (as is the case with our NCIFS below), the limit set $J(\Phi)$ is compact since it is an intersection of a decreasing sequence of compact sets.

To compute the Hausdorff dimension via Bowen's formula we will employ the following.

\begin{theorem}[Proposition 1.3 of \cite{RU}]\label{BowenFormula}
Suppose that $\Phi$ is a system such that both limits $$a:= \lim_{k \to \infty} \frac{1}{k} \log \# I^{(k)}$$ and $$b:= \lim_{k \to \infty, j \in I^{(k)}} \frac{1}{k} \log \left (1/\| D\phi_j^{(k)}\|\right )$$ exist and are finite and positive.  Then $HD(J(\Phi))=a/b$.
\end{theorem}

Note that the limit for $b$, when it exists, must exist independently of the choices of $j=j(k)$ taken from each $I^{(k)}$.  In our application, we will see that the quantities $\|D\phi_j^{(k)}\|$ will always be independent of $j$.

Our next step is to verify that we can obtain a NCIFS $\Phi$ whose limit set $J(\Phi)$ will be identical with $\J= \J_0$.  First we set $X : = \overline{\mathrm{D}}(0,2)$.  As noted in Remark~\ref{Remark1}(\ref{Remark1c}), each map $Q_{M_{k-1},M_k}(z)=z^{2^{m_k+1}}+c_k$ has the full set of $2^{m_k+1}$ branches of the inverse each defined on $\mathrm{D}(0,4) \supset X$.  For each fixed $k$, we denote this set of inverse functions by $\{\phi_j^{(k)}\}_{j=1}^{2^{m_k+1}}$, which we choose as our $\Phi^{(k)}$, noting then that $\#I^{(k)}=2^{m_k+1}$ in Definition~\ref{NCIFSdef}.  It then follows from the invariance condition Remark~\ref{Remark1}(\ref{Remark1c}) that each of the maps $\phi_j^{(k)}$, $ j=1, \ldots, 2^{m_k+1}$ maps the set $X$ into itself.

By Lemma~\ref{Lem2.2}(\ref{Lem2.2b}), we see that
\begin{equation}\label{phiDeriv}
\|(\phi_j^{(k)})'\|=\sup\{|{(\phi_j^{(k)}})'(x)|:x\in X\}=\frac{1}{2^{m_k+1}} (|c_k|-2)^{\left(\frac{1}{2^{m_k+1}}-1\right)}.
\end{equation}

Note that $\|(\phi_j^{(k)})'\|$ is in particular independent of $j$ and thus of the particular inverse branch used. Using the terminology given in Definition 4.1 on page 1993 of \cite{RU}, we can thus say our system $\Phi$ is \emph{balanced}.

We now quickly verify that conditions (A)-(D) of Definition~\ref{NCIFSdef} are met, thus giving that the associated $\Phi$ is indeed a NCIFS.

The open set condition (\ref{NCIFSa}) follows immediately from Remark~\ref{Remark1}(\ref{Remark1c}) (see Figure~\ref{Preimages} for an illlustration).  Note that the sets $X_k$ from Definition~\ref{LimitSetdef} are identical with the sets ${\mathcal S}_k$ in~\eqref{Sk}, and thus, by Remark~\ref{Invariance2}, are a union of $\prod_{i=1}^k 2^{m_i+1}=2^{M_k}$ mutually disjoint sets.
As noted in~\cite{RU}, for dimension $d=2$ the bounded distortion condition (\ref{NCIFSc}) follows from (\ref{NCIFSb}), shown below, and the standard distortion theorems for univalent functions, e.g., Theorem 1.6 of~\cite{CG}.
Since the maps $\phi_j^{(k)}$ send $X$ into itself, the uniform contraction condition (\ref{NCIFSd}) holds by Lemma~\ref{Lem2.2}(\ref{Lem2.2b}) with $\eta = 2^{-\frac{11}{4}}$.

It remains to show the conformality condition (\ref{NCIFSb}), which we establish using Lemma~\ref{Lem2.2}(\ref{Lem2.2b}) with $V=\mathrm{D}(0,2+\epsilon)$ for any small fixed $\epsilon >0$ such that $\eta_\epsilon<1$.  Fixing $k \ge 1$ and $j \in I^{(k)}$, gives that $\sup\{|{(\phi_j^{(k)}})'(x)|:x\in V\} \leq \eta_\epsilon$, which, combined with the convexity of $V$ and the fact that  $\phi_j^{(k)}(X) \subseteq X$, yields that each point of $\phi_j^{(k)}(V)$ must lie within a distance of $\eta_\epsilon \cdot \epsilon$ of $\phi_j^{(k)}(X) \subseteq X$, and so $\phi_j^{(k)}(V)\subseteq \overline{\mathrm{D}}(0,2+\eta_\epsilon \cdot \epsilon) \subseteq V$.

Before embarking on proving Theorems~\ref{ThmHDbound} and~\ref{ThmHDmax}, we remark that the limit set of the NCIFS $\Phi$ constructed above does indeed coincide with the Julia set $\J$, this being an immediate consequence of Theorem~\ref{ThmJm} and the fact that each $\S_k=X_k$.

\begin{proof}[Proof of Theorem~\ref{ThmHDbound}]
We prove the result for the case $m = 0$. Using part (a) of Proposition 3.3 of~\cite{Fal}, the result for the other iterated Julia sets follows from complete invariance (Theorem~\ref{ThmCompInvar}) and the fact that the polynomials $P_m$ are complex analytic and therefore $1$-H\"older.

For any $n \in \N$ and any $j \in I^{(n)}$, by~\eqref{phiDeriv} we see that, since $|c_n|>4$, we must have
$\|(\phi_j^{(n)})'\| \leq \frac{1}{2^{m_n+1}}$.
For all $k \in \N$ and $\omega= \omega_1 \dots \omega_k \in I^k$, we then see that $\phi_\omega=\phi_{\omega_1}^{(1)} \circ \dots \circ \phi_{\omega_k}^{(k)}$ satisfies $\|\phi_\omega'
\| \leq \frac{1}{2^{m_1+1}}\cdots \frac{1}{2^{m_k+1}}=\frac{1}{2^{M_k}}$.  Hence, by the convexity of $X$, $X_k$ is covered by $2^{M_k}$ sets $X_\omega= \phi_\omega (X)$ with diameters $\mathrm{diam}(X_\omega) \leq \frac{1}{2^{M_k}} \cdot \mathrm{diam}(X)=\frac{4}{2^{M_k}}$.

Fix $\delta >0$.  We then choose $k$ such that $\frac{4}{2^{M_k}}<\delta$, and note that, since $\J \subset X_k$, we have $\mathcal{H}_\delta^1 (\J) \leq 2^{M_k} \cdot \frac{4}{2^{M_k}}=4$.  Letting $\delta \to 0$, we see that the Hausdorff 1-dimensional measure satisfies $\mathcal{H}^1 (\J) \leq 4$, thus implying $\mathrm{HD}(\J) \leq 1$.
\end{proof}

\begin{proof}[Proof of Theorem~\ref{ThmHDmax}]

We first restrict ourself to the case where $m=0$ and show we can construct our sequence $\Pm$ so that $\HD(\J) = 1$.

Define two sequences of real numbers $\{a_k\}_{k=1}^\infty$, $\{b_k\}_{k=1}^\infty$ by

\begin{equation}\label{ak}
a_k:= \frac{1}{k} \log \# I^{(k)} = \frac{1}{k} \log 2^{m_k+1}= \frac{1}{k} (m_k+1)\log 2,
\end{equation}

and, using \eqref{phiDeriv},

 \vspace{-.5cm}
  \begin{eqnarray}
b_k &:=& \frac{1}{k}\log \frac{1}{\|(\phi_j^{(k)})'\|}= \frac{1}{k} \log \left(2^{m_k+1} (|c_k|-2)^{\left(1-\frac{1}{2^{m_k+1}}\right)}\right)  \\
  &=&\frac{1}{k} \left[ (m_k+1)\log 2 + \left(1-\frac{1}{2^{m_k+1}}\right)\log (|c_k|-2)\right]\\
  &=& a_k + \frac{1}{k} \left[\left(1-\frac{1}{2^{m_k+1}}\right)\log (|c_k|-2)\right].\label{bk}
\end{eqnarray}

Now we show that we can choose $\{m_k\}_{k=1}^\infty$ and $\{c_k\}_{k=1}^\infty$ satisfying~\eqref{Invariance*} with $|c_k| \to \infty$ (and each $|c_k| >4$) so $\mathcal{J}$ will be HNUP by Theorem~\ref{ThmDichotomy}.

By \eqref{ak} and \eqref{bk}, we see that by ensuring

\vspace{-.2cm}
\begin{enumerate}[(i)]
\item \label{Cond1} $\lim_{k \to \infty} \frac {m_k}k$ exists as a finite and positive number, and

\item \label{Cond2} $\lim_{k \to \infty} \frac{\log |c_k|}{k}=0$,

\end{enumerate}
it follows that $\{a_k\}$ and $\{b_k\}$ are convergent with the same finite and positive limit.  Thus, we may apply Theorem~\ref{BowenFormula} to conclude that $HD(J(\Phi))=\mathrm{HD}(\J)=1$.

 To see this can indeed happen, for each $k \ge 1$, we set $c_k = k + 4$ and $m_k = k+1$. One can then check readily that the invariance condition \eqref{Invariance*} is satisfied for all $k$. It is also easy to verify that both of the conditions~(\ref{Cond1}) and~(\ref{Cond2}) above are met, whence the result follows.

We now complete the proof by considering an arbitrary $m_0>0$.  Choose some $M_k>m_0$.  By the complete invariance shown in Theorem~\ref{ThmCompInvar}, we have $Q_{m_0, M_k}(\J_{m_0})=\J_{M_k}$.  As was done in last part of the proof of Theorem~\ref{ThmDichotomy}, we apply the above argument to the truncated sequence $\{P_m\}_{m=M_k+1}^\infty$ to show $\HD(\J_{M_k}) = 1$.  Again applying part (a) of Proposition 3.3 in~\cite{Fal} for the $1$-H\"older map $Q_{m_0, M_k}$, we then must have $1=\HD(\J_{M_k}) \leq \HD(\J_{m_0}) \leq 1$, where the last inequality follows from  Theorem~\ref{ThmHDbound}.
\end{proof}

\section*{Acknowledgments} This work was partially supported by a grant from the Simons Foundation (\#318239 to Rich Stankewitz).

The third author (Hiroki Sumi) was partially supported by JSPS Grant-in-Aid for Scientific Research (B) Grant number JP 19H01790.

%

\end{document}